\documentclass[a4paper, 11pt, reqno]{amsart}
\usepackage{amsmath}
\usepackage{amsfonts}
\usepackage{amssymb}
\usepackage{graphicx}
\usepackage{amsthm}

\newtheorem{theorem}{Theorem}

\newtheorem{definition}[theorem]{Definition}

\newtheorem{lemma}[theorem]{Lemma}

\newtheorem{proposition}[theorem]{Proposition}
\newtheorem{remark}[theorem]{Remark}

\makeatletter
\newcommand{\addresseshere}{%
	\enddoc@text\let\enddoc@text\relax
}
\makeatother

\newcommand{\pr}{\partial_{r}}

\newcommand{\ric}{\mathrm{Ric}}
\newcommand{\dive}{\mathrm{div}}

\newcommand{\s}{\sigma}

\usepackage[margin=1in]{geometry}

\begin{document}
\title{Serrin's type problems in warped product manifolds}
\author{Alberto Farina and Alberto Roncoroni}
\thanks{}
\address{Alberto Farina, LAMFA, CNRS UMR 7352, Universit\'e de Picardie Jules Verne, 33 rue Saint-Leu, 80039 Amiens Cedex 1, France}
\email{alberto.farina@u-picardie.fr}
\address{Alberto Roncoroni, Dipartimento di Matematica e Informatica ``Ulisse Dini'', Universit\`a degli Studi di Firenze, Viale Morgagni 67/A, 50134 Firenze, Italy}
\email{alberto.roncoroni@unifi.it}

\date{\today}

\begin{abstract}
In this paper we consider Serrin's overdetermined problems in warped product manifolds and we prove Serrin's type rigidity results by using the $P$-function approach introduced by Weinberger. 
\end{abstract}

\maketitle

\section{Introduction}


In \cite{Serrin}, J. Serrin proved the following celebrated result: \textit{if there exists a positive solution $u\in C^2(\bar{\Omega})$ to the following semilinear overdetermined problem 
\begin{equation}\label{Serrin1}
\begin{cases}
\Delta u +f(u)=0 &\mbox{in } \Omega, \\ 
u= 0 &\mbox{on } \partial\Omega, \\
\partial_\nu u=c &\mbox{on } \partial\Omega\, ,
\end{cases}
\end{equation}
where $\Omega\subset\mathbb{R}^n$ is a bounded domain with boundary of class $C^2$, $f \in C^1$
and $\nu$ denotes the unit normal to $\partial\Omega$, then $\Omega$ must be a ball and $u$ is radially symmetric.}
This result is known as the Serrin's symmetry (or rigidity) result. The technique introduced in \cite{Serrin} to prove this result is a refinement of the famous reflection principle due
to Alexandrov (see \cite{Alex}) and is the so-called moving planes method together with the maximum principle and a new version of the Hopf's boundary point Lemma (the Corner lemma, see \cite[Lemma 1]{Serrin}). We mention that the technique of Serrin applies to more generally uniformly elliptic operators (see \cite{Serrin}) and it has inspired the study of various properties and symmetry results for positive solutions of elliptic partial differential equations in bounded and unbounded
domains (see for instance the seminal paper \cite{GNN}).

In \cite{Weinberger}, H. Weinberger provided a simpler proof in the case $\Delta u=-1$ based on what are nowadays called P-function and using integral identities. We mention that also the approach of Weinberger inspired several works in the context of elliptic
partial differential equations (see e.g. \cite{CGS,FK,FV,FGK,GL,Payne,Sperb} and their references).

In literature there are generalizations of Serrin's result for domains in the so-called space forms, i.e. complete, simply connected Riemannian manifolds with constant sectional curvature. Thanks to the Killing-Hopf theorem (see \cite{Hopf,Killing}) it is well-known that space forms are isometric to the Euclidean space $\mathbb{R}^n$, to the hyperbolic space $\mathbb{H}^n$ or to the sphere $\mathbb{S}^n$. In particular in \cite{Kumaresan-Prajapat} and \cite{Molzon} the moving planes method is used to prove the analogue of Serrin's result for the problem \eqref{Serrin1}  for bounded domains in $\mathbb{H}^n$ and in $\mathbb{S}^n_+$ (we mention that in $\mathbb{S}^n$ the theorem is not true, see e.g. \cite{FMW} and \cite{Shklover}). Of particular interest for us are the work \cite{Ciraolo_Vezzoni} and \cite{Ron} where a $P$-function approach is used to prove the analogue of Serrin's theorem in space forms for the following equation 
\begin{equation}\label{Alexandrov}
\Delta u+nku=-1\, ,
\end{equation} 
where $k=0$ in $\mathbb{R}^n$, $k=1$ in $\mathbb{S}^n_+$ and $k=-1$ in $\mathbb{H}^n$ (see also \cite{Ciraolo_Roncoroni} where the same problem is considered for overdetermined problems in convex cones of a space form).
%
Inspired by these results, in this paper we prove the analogue of Serrin's theorem in a particular class of Riemannian manifolds: the warped product manifolds. The warped products (defined firstly in \cite{BN}, see also \cite{Libro_warped}) are the most fruitful generalization of the notion of Cartesian (or direct) product and of the notion of rotationally symmetric manifold (the one considered in \cite{Ciraolo_Vezzoni} and \cite{Ron}). We recall their definition

\begin{definition}\label{warped}
A Riemannian manifold $(M,g)$ of dimension $n\geq 2$ is a warped product manifold if
\begin{equation}
M=I\times N \quad \text{ and } \quad g=dr\otimes dr +\s^2(r)g_N\, ,
\end{equation}
where
\begin{itemize}
\item $I$ is an open interval;
\item $(N,g_N)$ is a smooth and connected $(n-1)$-dimensional Riemannian manifold without boundary;
\item $\s:I\rightarrow\mathbb{R}$ is a smooth function such that:
$$ \s >0 \quad on \quad I.
$$
\end{itemize}
\end{definition}
In particular $(M,g)$ is a smooth connected $n$-dimensional manifold without boundary (not necessarily complete). 

  
Our first result concerns warped product manifolds where the fiber manifold $(N,g_N)$ satisfies 
\begin{equation}\label{Ricci_N}
\ric_N  \geq (n-2)\rho g_N \quad  \text{ for some constant $\rho \in \mathbb{R}$} \, ,
\end{equation}
and the function $\sigma$ is given by 
\begin{equation}\label{warped_function}
\s(r)=\begin{cases}
c_1 e^{\sqrt{-k}r}+c_2 e^{-\sqrt{-k}r} &\mbox{if } k<0, \\ c_1+c_2 r &\mbox{if } k=0, \\ c_1 \cos(\sqrt{k}r)+c_2\sin(\sqrt{k}r) &\mbox{if } k>0,
\end{cases}
\end{equation}
where the constants $c_1$ and $c_2$ are chosen so that $\sigma>0$ on $I$, 
\begin{equation}\label{constants}
	\s'(r)\geq 0 \quad \text{ for all $r\in I$ and $\sigma'\not\equiv0$}
\end{equation}
and the warped product manifold $(M,g)$ satisfies 
\begin{equation}\label{Ricci_primo}
\ric_M \geq (n-1) k g. 
\end{equation}

The are a lot of manifolds satisfying those properties. Indeed, the Ricci tensor of a warped product manifold is given by the following expression 
\begin{equation}\label{Ricci_warped}
\ric_M= \ric_N - (\s(r) \s''(r) + (n-2) \s'(r)^2)g_N - (n-1) \frac{\s''(r)}{\s(r)} dr\otimes dr
\end{equation}
(see e.g. \cite{Brendle,Besse}) and so
\begin{equation}\label{Ricci_warped2}
\begin{aligned}
\ric_M  =& (\ric_N - (n-2)\rho g_N) -\left(\dfrac{\s''(r)}{\s(r)}-(n-2)\dfrac{\rho-\s'(r)^2}{\s(r)^2}\right)g + \\ &-(n-2)\left(\dfrac{\s''(r)}{\s(r)}+\dfrac{\rho-\s'(r)^2}{\s(r)^2}\right)dr\otimes dr \\
 \geq & -\left(\dfrac{\s''(r)}{\s(r)}-(n-2)\dfrac{\rho-\s'(r)^2}{\s(r)^2}\right)g -(n-2)\left(\dfrac{\s''(r)}{\s(r)}+\dfrac{\rho-\s'(r)^2}{\s(r)^2}\right)dr\otimes dr\,,
\end{aligned}
\end{equation}
where in that letter we have used \eqref{Ricci_N}.
Now, if the fiber manifold $(N,g_N)$ satisfies \eqref{Ricci_N} with $ \rho >0$, 
the open interval $I$ satisfies 
\begin{equation}\label{intervallo-warped}
\begin{cases}
I \subset (0,\infty) &\mbox{if } k\leq 0, \\ 
I \subset (0, \frac{\pi}{2\sqrt{k}}) &\mbox{if } k>0,
\end{cases}
\end{equation}
and the warping function is given by 
\begin{equation}\label{warped_function2}
\s(r)=\begin{cases}
\sqrt{\rho} \frac{\sinh(\sqrt{-k}r)}{\sqrt{-k}} &\mbox{if } k<0, \\ 
\sqrt{\rho} r &\mbox{if } k=0, \\ 
\sqrt{\rho} \frac{\sin(\sqrt{k}r)}{\sqrt{k}}&\mbox{if } k>0,
\end{cases}
\end{equation}
then, thanks to \eqref{Ricci_warped2}, the resulting warped product manifold $(M,g)$ satisfies \eqref{Ricci_primo}. 
In particular, if $N = \mathbb{S}^{n-1}$ is the unit sphere endowed with its canonical metric, we recover the case (of open sets) of the space forms with constant sectional curvature equal to $k$.

Other non-trivial examples of warped product manifolds satisfying our assumptions are obtained by choosing $N$ satisfying \eqref{Ricci_N} and $I$ and $ \sigma$ as follows:  
\begin{equation}\label{ex1}
I \subset \mathbb{R}\, , \quad \s(r)=e^{\sqrt{-k} r}\,\, , \text{$k <0$ and $ \rho = 0$}\, ,
\end{equation}
or
\begin{equation}\label{es2}
I \subset (0, \infty)\, , \quad  \s(r)=c_1 e^{\sqrt{-k}r}+c_2 e^{-\sqrt{-k}r}\, ,  \text{$k <0$, $c_1 \geq c_2 >0$ such that $\rho=4kc_1 c_2 (<0)$}\, . 
\end{equation}
Also observe that, when \eqref{warped_function2} or \eqref{ex1} or \eqref{es2} are in force and the fiber manifold $(N,g_N)$ is Einstein (i.e., it satisfies \eqref{Ricci_N} with the equality sign) then also the resulting warped product manifold $(M,g)$ is Einstein (i.e., it satisfies \eqref{Ricci_primo} with the equality sign).

\medskip

With these preliminaries, the first result of this paper is the following. 
\medskip

\begin{theorem}\label{teo_princ}
Let $(M,g)$ be a warped product manifold (not necessarily complete) such that \eqref{Ricci_N}, \eqref{warped_function}, \eqref{constants} and \eqref{Ricci_primo} hold true. Let $\Omega\subset M$ be a domain (i.e. open and connected set) with boundary of class $C^1$ such that $\overline{\Omega}$ is compact. Let $u\in C^3(\Omega)\cap C^2(\overline{\Omega})$ be a solution to 
\begin{equation}\label{pb_principale}
\begin{cases}
\Delta u + nku=-1 &\mbox{in } \Omega, 
\\ u>0 &\mbox{in } \Omega,
\\ u= 0 &\mbox{on } \partial\Omega\, , 
\end{cases}
\end{equation}
such that, for some constant $c$,
\begin{equation}\label{cond_sovra}
|\nabla u|=c \quad \text{on $\partial\Omega$.}
\end{equation}
Then $\Omega$ is a metric ball and $u$ is a radial function, i.e. $u$ depends only on the distance from the center of the ball.
\end{theorem}

The second result of this paper is a generalization of Theorem \ref{teo_princ} in general warped product manifolds, i.e., for general warping functions $\sigma>0$,  where we assume a compatibility condition between the geometry of the warped product and the solution to the overdetermined problem.


\begin{theorem}\label{teo_princ_bis}
 Let $(M,g)$ be a warped product manifold (not necessarily complete) such that
 \begin{equation}\label{hp_ricci}
\ric_M\geq (n-1)k g\, , \quad \text{ for some $k\in\mathbb{R}$,}
 \end{equation}
 and 
\begin{equation}\label{hp_warped}
	\s'(r)\geq 0 \quad \text{ for all $r\in I$, $\sigma'\not\equiv0$.}
\end{equation} 
Let $\Omega\subset M$ be a domain 
with boundary of class $C^1$ such that $\overline{\Omega}$ is compact. Let $u\in C^3(\Omega)\cap C^2(\overline{\Omega})$ be a solution to \eqref{pb_principale}-\eqref{cond_sovra}. If $u$ satisfies the following compatibility condition 
	\begin{equation}\label{comp}
	\int_{\Omega}\left(k\sigma'+\dfrac{(\sigma''\sigma^{n-1})'}{n\sigma^{n-1}} \right)u^2\geq 0\, .
	\end{equation}
Then $\Omega$ is a metric ball and $u$ is a radial function, i.e. $u$ depends only on the distance from the center of the ball.
\end{theorem}

Observe that in Theorem \ref{teo_princ_bis} we do not assume \eqref{Ricci_N} and observe that if we assume that $\sigma$ is given by \eqref{warped_function}, then conditions \eqref{hp_ricci} and \eqref{hp_warped} are trivially satisfied, moreover also the compatibility condition is trivially satisfied.

This result improves the result in \cite{Ron} where the case of model manifolds and $k=0$ was considered. More precisely, in \cite{Ron} the problem is $\Delta u=-1$ on model manifolds with nonnegative Ricci curvature and such that $\sigma'>0$, moreover there is a compatibility condition which is exactly \eqref{comp} with $k=0$ (see \cite[Formula 3]{Ron}). The conclusion in \cite{Ron} is stronger than the one in Theorem \ref{teo_princ_bis}, indeed in \cite{Ron} the domain $\Omega$ is a ball and the metric in this ball is the Euclidean one.
This strong rigidity result will be generalized in our next result (see Theorem \ref{teo_princ_ter} in the sequel). 
To this end we observe that the model manifolds can be written as warped product manifolds with a pole, explicitly one takes: 
$$
I=[0,R) \quad \text{ with $R\leq+\infty$} \qquad \text{and} \quad N=\mathbb{S}^{n-1}\,\, \text{(endowed with its canonical metric)},
$$
together with the right hypothesis on the function $\sigma$ which makes the metric $g$ in \eqref{warped} smooth (see \cite[Definition 2]{Ron}). 
We are now in position to state our third and last result. It deals with model manifolds and provides a full generalization of \cite[Theorem 3]{Ron}. In particular, under the assumption that the pole of the model $o$ is inside the domain, we prove that the domain must be a geodesic ball around $o$ and that the metric in the ball has constant sectional curvatures.

\begin{theorem}\label{teo_princ_ter}
	Let $(M,g)$ be a model manifold (not necessarily complete) such that \eqref{hp_ricci} and \eqref{hp_warped} hold true. Let $\Omega\subset M$ be a domain 
	with boundary of class $C^1$ such that $\overline{\Omega}$ is compact. We assume that $o\in\Omega$. Let $u\in C^3(\Omega)\cap C^2(\overline{\Omega})$ be a solution to \eqref{pb_principale}-\eqref{cond_sovra}. If $u$ satisfies the compatibility condition \eqref{comp}. Then $\Omega$ is a metric ball centred at $o$ of radius $\rho$ and $u$ is a radial function given by
	\begin{equation}
	u(r)=\begin{cases}
	\frac{\cosh(\sqrt{-k}r)}{kn\cosh(\sqrt{-k}\rho)}-\frac{1}{nk} &\mbox{if } k<0\, , \\ \frac{\rho^2}{2n}-\frac{r^2}{2n} &\mbox{if } k=0, \\
	\frac{\cos(\sqrt{k}r)}{kn\cos(\sqrt{k}\rho)}-\frac{1}{nk} &\mbox{if } k>0 \, ,
	\end{cases}
	\end{equation}
	where $r$ is the geodesic distance from $o$.
	
	Moreover, the warping function $\sigma$ in the metric ball is given by the following expression:
	\begin{equation}\label{metric}
	\sigma(r)=\begin{cases}
	 \frac{\sinh(\sqrt{-k}r)}{\sqrt{-k}} &\mbox{if } k<0, \\ 
	 r &\mbox{if } k=0, \\ 
	 \frac{\sin(\sqrt{k}r)}{\sqrt{k}}&\mbox{if } k>0\, .
	\end{cases}
	\end{equation}
\end{theorem}


Observe that in Theorem 4 if we assume that $\sigma$ is given by \eqref{metric} then \eqref{hp_ricci},\eqref{hp_warped} and the compatibility condition \eqref{comp} are automatically satisfied.

Before we get to the heart of the paper we comment about our Theorems.

\medskip

Theorems \ref{teo_princ} and \ref{teo_princ_ter} recover and improve the result in \cite{Ciraolo_Vezzoni} where the domain $\Omega$ was a bounded domain of one of the following three models : the Hyperbolic space $\mathbb{H}^n$, the Euclidean space $\mathbb{R}^n$ and the hemisphere $\mathbb{S}^n_+$. Indeed, these manifolds are model manifolds (take $I =[0,+\infty)$ for $\mathbb{H}^n$ and $\mathbb{R}^n$, $I = [0, \frac{\pi}{2})$ for $\mathbb{S}^n_+$ and $\sigma $ as in \eqref{metric}, with $ k \in \{-1,0,1\}$) and so we recover the results in \cite{Ciraolo_Vezzoni}  by applying Theorem \ref{teo_princ_ter}, if the pole belongs to $\Omega$, and by applying Theorem \ref{teo_princ} when the pole does not belong to $\Omega$ (in the latter we have seen $\Omega$ as a bounded domain of the warped product manifolds : $\mathbb{H}^n$ minus one point, $\mathbb{R}^n$ minus one point and $\mathbb{S}^n_+$ minus one point (the pole)).




Moreover, the authors of \cite{Ciraolo_Vezzoni} do not require that the solution of \eqref{pb_principale} is positive, indeed in the cases of $\mathbb{R}^n$ and of $\mathbb{H}^n$ this follows from the standard maximum principles, while in the case of $\mathbb{S}^n_+$ this follows from the fact that the first eigenvalue of the Dirichlet Laplacian on the hemisphere is $n$ and the corresponding eigenfunction is strictly positive (see \cite[Proof of Lemma 2.4]{Ciraolo_Vezzoni} for details).
This proves that the main result in \cite{Ciraolo_Vezzoni} is a special case of our Theorem \ref{teo_princ}  and Theorem \ref{teo_princ_ter}.

\medskip	

A case of interest related to our Theorem \ref{teo_princ_bis} and \ref{teo_princ_ter} is the following example. Consider $I=(a,b+\varepsilon)$, with $ 0 < a < b < \infty$, $ \varepsilon >0$, $N$ such that \eqref{Ricci_N} holds with $\rho \geq1$ and 
$$
\s(r)=\begin{cases}
r &\mbox{in } (a,b] , \\ r(1-e^{-\frac{1}{r-b}}) &\mbox{in } (b,b+\varepsilon) \, .
\end{cases}
$$
The metric given by $dr\otimes dr+\s^2(r)g_N$ is smooth and the corresponding warped product is a non complete Riemannian manifold (a rotationally symmetric smooth manifold if $N = \mathbb{S}^{n-1}$ is the unit sphere endowed with its canonical metric).
Moreover the warped product satisfies the hypothesis of Theorem \ref{teo_princ_bis}, indeed from a direct computation by making use of formula \eqref{Ricci_warped} we have, for $\varepsilon$ sufficiently small,
\begin{equation}
\begin{cases}
\ric_M =0 &\mbox{for any point in } (a,b] \times N, \\ 
\ric_M \geq 0 &\mbox{for any point in } (b, b+\varepsilon) \times N,\\
\end{cases}
\end{equation}
hence \eqref{hp_ricci} holds with $k=0$ and \eqref{hp_warped} is trivially satisfied. Now, the idea is to take the domain $\Omega$ in $(a,b)\times N$, where the metric is the Euclidean one, and so also \eqref{comp} is trivially satisfied. So, summing up, we have that this is an non trivial example in which our Theorem \ref{teo_princ_bis} works. 

This example also applies in the case of Theorem \ref{teo_princ_ter} (just consider the model manifold obtained by taking $I=[0,b+\varepsilon)$).

\medskip

The proof of all results is based on the following $P$-function (considered also in \cite{Weinberger} in \cite{Ciraolo_Vezzoni} and in \cite{Ron})
\begin{equation}\label{P-funzione}
P(u):=|\nabla u|^2+\dfrac{2}{n}u+ku^2 \, ,
\end{equation}
where $u$ is the solution to \eqref{pb_principale}-\eqref{cond_sovra}.

\medskip

\noindent {\bf Organization of the paper.} The paper is organized as folows: in Section \ref{Preliminaries} we prove general results related to warped product manifolds, in Section \ref{Principal_section} we prove the main results of the paper and in Appendix \ref{Appendix} we prove a general property of about the star-shapedness of geodesic balls inside a general Riemannian manifold.

\bigskip

\noindent {\bf Acknowledgements.} The authors wish to thank Stefano Pigola and Luigi Vezzoni for useful discussions. A.R. has been partially supported by the Gruppo Nazionale per l'Analisi Matematica, la Probabilit\'a e le loro Applicazioni (GNAMPA) of the Istituto Nazionale di Alta Matematica (INdAM). This manuscript was started while A.R. was visiting the LAMFA, Universit\'e de Picardie Jules Verne in Amiens, which is acknowledged for the hospitality.

\section{Preliminaries}\label{Preliminaries}

In this section we collect preliminaries results that we will use to prove our Theorems. In Subsection \ref{sub_Obata} we prove an Obata-type rigidity result in a very general context. In Subsection \ref{sub_Poho} we prove a Pohozaev-type identity which is the fundamental ingredients in the proofs of all Theorems.

\subsection{An Obata-type result}\label{sub_Obata}

A key point in the proof is that the $P$-function \eqref{P-funzione} is subharmonic accrodingly to the following lemma, moreover we are able to characterize the harmonicity of $P$; this is a very general result which holds true in every Riemannian manifold with Ricci curvature bounded from below.

\begin{lemma}\label{pre1}
	Let $(M,g)$ be an $n$-dimensional Riemannian manifold (not necessarily complete) such that
	\begin{equation}\label{Ricci_general}
	\ric_M\geq (n-1)kg\quad \text{ for $k\in\mathbb{R}$.}
	\end{equation}
	Let $\Omega\subset M$ be a domain and let $u\in C^2(\Omega)$ be a solution to 
	\begin{equation}\label{Delta u}
	\Delta u + nku=-1 \quad \text{ in $\Omega$.} 
	\end{equation}
	Then 
	$$
	\Delta P(u)\geq 0 \quad \text{ in $\Omega$.} 
	$$
	where $P$ is given by \eqref{P-funzione}. Moreover, 
	$$
	\Delta P(u)= 0
	$$ 
	if and only if 
	\begin{equation}\label{harm1}
	\nabla^2 u=-\left(\dfrac{1}{n}+ku\right)g \, \quad \text{in $\Omega$,}
	\end{equation}
	and 
	\begin{equation}\label{harm2}
	\ric_M(\nabla u,\nabla u)=(n-1)k|\nabla u|^2\, \quad \text{in $\Omega$;}
	\end{equation}
	where $\nabla^2$ denotes the Hessian matrix and $|\cdot|^2=g(\cdot,\cdot)$.
\end{lemma}

\begin{proof}
	From the Bochner-Weitzenboch formula and Cauchy-Schwarz inequality we get 
	\begin{equation*}
	\begin{aligned}
	\Delta|\nabla u|^2=&2|\nabla^2 u|^2+2g(\nabla(\Delta u),\nabla u)+2\ric_M(\nabla u,\nabla u)\\
	\geq & \dfrac{2}{n}(\Delta u)^2+2g(\nabla(\Delta u),\nabla u)+2\ric_M(\nabla u,\nabla u)\, .
	\end{aligned}
	\end{equation*}
	From \eqref{Ricci_general} we obtain
	\begin{equation}\label{Cat}
	\begin{aligned}
	\Delta|\nabla u|^2\geq \dfrac{2}{n}(\Delta u)^2+2g(\nabla(\Delta u),\nabla u)+2(n-1)k|\nabla u|^2\, .
	\end{aligned}
	\end{equation}
	Since $\Delta u=-1-nku$, then \eqref{Cat} becomes
	\begin{equation}
	\begin{aligned}
	\Delta|\nabla u|^2&\geq \dfrac{2}{n}\Delta u(-1-nku)+2g(\nabla(-1-nku),\nabla u)+2(n-1)k|\nabla u|^2 \\
	&=-\dfrac{2}{n}\Delta u-2ku\Delta u -2k|\nabla u|^2 \\
	&=-\dfrac{2}{n}\Delta u-k\Delta(u^2)\, ,
	\end{aligned}
	\end{equation} 
	where in the last equality we use the fact that $\Delta(u^2)=2|\nabla u|^2+2u\Delta u$. Hence $\Delta P(u)\geq 0$. From the argument above it is clear that $\Delta P(u)= 0$ if and only if \eqref{harm2} holds and the equality in Cauchy-Schwarz inequality holds, i.e.
	$$
	n|\nabla^2 u|=(\Delta u)^2 
	$$
	which, by using \eqref{Delta u}, implies \eqref{harm1}.
\end{proof}

In order to prove the $P$ is harmonic we will use a Pohozaey-type identity (see Proposition \ref{Pohozaev} below); once \eqref{harm1} holds then the conclusions of Theorems \ref{teo_princ}, \ref{teo_princ_bis} and \ref{teo_princ_ter} will follow. In the next lemma we prove that if $u$ solves \eqref{harm1} in $\Omega$, then $\Omega$ must be a metric ball and $u$ is a radial function, i.e. $u$ depends only on the distance from the center of the ball.

\begin{lemma}\label{Obata_lemma}
	Let $(M,g)$ be an $n$-dimensional Riemannian manifold (not necessarily complete)
	and let $\Omega$ be a domain in $M$ such that $\overline{\Omega}$ is compact. 	
Assume that there exists a function $u:\bar{\Omega}\rightarrow\mathbb{R}$ such that $u\in C^0(\bar{\Omega})\cap C^2(\Omega)$ and is a solution to
	\begin{equation}\label{Obata}
	\begin{cases}
	\nabla^2 u=-\left(\dfrac{1}{n}+ku\right)g &\mbox{in } \Omega, \\ u>0 &\mbox{in } \Omega,\\ u= 0 &\mbox{on } \partial\Omega\, ,
	\end{cases}
	\end{equation}
	where $k\in\mathbb{R}$. Then $\Omega$ is a metric ball and $u$ is a radial function, i.e. $u$ depends only on the distance from the center of the ball.
\end{lemma}

The previous result can be seen as a local version of the well-known Obata-Tashiro-Kanai's results (see \cite{Obata,Tashiro,Kanai}). 

\begin{proof}[Proof of Lemma \ref{Obata_lemma}]
	Being $u > 0$ in $\Omega$ and $\overline{\Omega}$ compact we have that $u$ achieves its maximum at a point $p\in\Omega$, with $u(p)> 0$. Since $\overline{\Omega}$ is compact, we consider the open ball $B_g(p,\rho)$ 
	centered at $p$ and of radius $\rho=dist(p,\partial\Omega)>0$ and we call $q\in\partial\Omega$ a point that realizes this distance. Take $q_1$ and $q_2$ in $\partial B_g(p,\rho')$ where $0<\rho'<\rho$. By Lemma \ref{lemmaGeod1} in Appendix \ref{Appendix}, for $i=1,2$, there is a unit speed minimizing geodesic  $\gamma_i:[0,\rho']\rightarrow B_g(p,\rho')$ such that $\gamma_i(0)=p$ and $\gamma_i(\rho')=q_i$.  
	
	For $i=1,2$, we set $f_i(s)=u(\gamma_i(s))$ and so, from the first equation in \eqref{Obata} and the fact that $ \overline {B_g(p,\rho')} \subset \Omega$, we have that $f_i$ satisfies 
	\begin{equation}\label{Obata_geodetica}
	f_i''(s)=-\dfrac{1}{n}-kf_i(s)\, ,
	\end{equation}
	moreover
	\begin{equation}\label{Obata_geodetica-cond-in}
	f_i'(0)=0 \,  \quad  \text{ and } \quad f_i(0)=u(p)\, .
	\end{equation}
	Therefore, 
	$u$ has the same expression along $\gamma_1$ and $\gamma_2$, hence $u$ is constant on the spheres $\partial B_g(p,\rho')$ for any $0<\rho'<\rho$.
	
	Now we take $q_1 \in \partial B_g(p,\rho)$, we recall that also the point $ q $ belongs to $ \partial B_g(p,\rho)$, we set $ q_2 =q$ and we apply Proposition \ref{LemmaGeod2} in Appendix \ref{Appendix} to find, for $i=1,2$, the existence of  a unit speed minimizing geodesic  $\gamma_i:[0,\rho]\rightarrow B_g(p,\rho)$ such that $\gamma_i(0)=p$, $\gamma_i(\rho)=q_i$ and $ \gamma_i (t) \in B_g(p,\rho)$ for any $t \in [0, \rho).$
	Therefore, arguing as before, the functions $f_i(s)=u(\gamma_i(s))$ satisfy \eqref{Obata_geodetica} on $ [0, \rho) $  together with \eqref{Obata_geodetica-cond-in}. The continuity of $f_i$ on $ [0,\rho]$ then implies $ u(q_1) = u(q_2) = u(q) =0$, since $ q \in \partial \Omega$.  The latter proves that $u \equiv0$ on $\partial B_g(p,\rho)$ and so we conclude that $\Omega$ must be equal to $B_g(p,\rho)$. Since we already know that $u$ is constant on the spheres $\partial B_g(p,\rho') $, for any $0 <\rho' <\rho$, we also proved that $u$ depends only on the distance from $p$.
		
\end{proof}


\subsection{A Pohozaev-type identity}\label{sub_Poho}

As already mentioned we will use the following Pohozaev-type identity (similar results can be found in \cite{Ciraolo_Vezzoni, Ron}) which holds on general warped product manifolds

\begin{proposition}\label{Pohozaev}
	Let $(M,g$) be a warped product manifold (not necessarily complete). Let  $\Omega\subset M$ be a bounded domain with boundary of class $C^1$ and let $u\in C^3(\Omega)\cap C^2(\overline{\Omega})$ be a solution to
	\begin{equation}\label{pb_principale_bis}
	\begin{cases}
	\Delta u + nku=-1 &\mbox{in } \Omega, \\ u= 0 &\mbox{on } \partial\Omega\, , 
	\end{cases}
	\end{equation}
	where $k\in\mathbb{R}$ and such that, for some constant $c$,
	\begin{equation}\label{cond_sovra_bis}
	|\nabla u|=c \quad \text{on $\partial\Omega$.}
	\end{equation}
	Then the following identity
	\begin{equation}\label{pre2}
	\dfrac{(n+2)}{n}\int_{\Omega}\sigma' u=c^2\int_{\Omega}\sigma'+\dfrac{n-2}{2n}\int_{\Omega}\dfrac{(\sigma''\sigma^{n-1})'}{\sigma^{n-1}}u^2-2k\int_{\Omega}\sigma' u^2
	\end{equation}
	holds true.
\end{proposition}

Before proving it, we recall that the Laplace-Beltrami operator $\Delta$ on a warped product manifold $(M,g)$ acts on $C^2$-functions $u:M\rightarrow\mathbb{R}$ as follows:
\begin{equation}\label{Laplace-Beltrami}
\Delta u=\partial^2_r u+ (n-1)\dfrac{\s'}{\s}\partial_r u+\dfrac{1}{\s^2}\Delta_N u\, ,
\end{equation}
where $\Delta_N$ denotes the Laplace-Beltrami operator on the Riemannian manifold $(N,g_{N})$. Using this representation of the Laplace-Beltrami operator we can easily compute $\Delta(\s\partial_ru)$. 
\begin{lemma}\label{lemma1}
	Under the assumptions of Proposition \ref{Pohozaev} 
	the following formula holds 
	\begin{equation}\label{eq2}
	\Delta(\s \, \pr u) =  \s\, \pr \Delta u + 2\s'\, \Delta u + (2-n) \s''\, \pr u.
	\end{equation}
	true for every $C^3$-function $u:\Omega\rightarrow\mathbb{R}$.

	\begin{proof}[Proof of Lemma \ref{lemma1}]
		We compute 
		\begin{align*}
		\s \pr (\Delta u) =& \s \left\lbrace\pr^{3}u+(n-1)\dfrac{\s''\s-(\s')^{2}}{\s^{2}}\pr u+(n-1)\dfrac{\s'}{\s}\pr^2 u-2\dfrac{\s'}{\s^3} \Delta_N u + \dfrac{1}{\s^2}\pr(\Delta_N u)\right\rbrace\\
		=&\s\pr^3 u+(n-2)\s''\pr u+\s''\pr u-2(n-1)\dfrac{(\s')^2}{\s}\pr u+(n-1)\dfrac{(\s')^2}{\s}\pr u+\\
		&+(n+1)\s'\pr^2 u-2\s'\pr^2 u -2\dfrac{\s'}{\s^2}\Delta_N u+\dfrac{1}{\s}\pr(\Delta_N u)+\dfrac{1}{\s^2}\Delta_N(\s\pr u)-\dfrac{1}{\s^2}\Delta_N(\s\pr u)\\
		=& \Delta(\s\pr u)+(n-2)\s''\pr u-2\s'\Delta u+\dfrac{1}{\s}\pr(\Delta_N u)-\dfrac{1}{\s^2} \Delta_N(\s\pr u), 
		\end{align*}
		i.e.
		\begin{equation*}
		\Delta(\s\pr u)=\s \pr (\Delta u)+(2-n)\s''\pr u+2\s'\Delta u.
		\end{equation*}
	\end{proof}
	
\end{lemma} 

\noindent Formula \eqref{eq2} is a key ingredient in order to prove the Pohozaev-type identity.

\noindent Moreover, we observe that condition \eqref{cond_sovra_bis}, implies that the unit exterior normal to $\partial\Omega=\{u=0\}$ is given by
$$
\nu=-\dfrac{\nabla u}{|\nabla u|}|_{\partial\Omega}\, ;
$$
hence
\begin{equation}\label{der_norm}
\partial_\nu u=g(\nabla u,\nu)=-c \quad \text{on $\partial\Omega$.}
\end{equation}

Now we can prove Proposition \ref{Pohozaev}.

\begin{proof}[Proof of Proposition \ref{Pohozaev}]
	From Lemma \ref{lemma1} and since $\Delta u=-1-nku$ we get 
	\begin{equation}\label{Laplaciano}
	\Delta(\sigma\partial_r u)=-nk\sigma\partial_r u+(2-n)\sigma''\partial_r u-2\sigma'-2nk\sigma'u\, .
	\end{equation}
	From \eqref{Laplaciano} and, by using again that $\Delta u=-1-nku$, we have 
	\begin{equation}\label{America}
	(2-n)\int_{\Omega}\sigma''\partial_r u u -2\int_{\Omega}\sigma' u - 2nk\int_{\Omega}\sigma' u^2+\int_{\Omega}\sigma\partial_r u=\int_{\Omega}\left(\Delta(\sigma\partial_r u) u-\sigma\partial_r u\Delta u\right)\, .
	\end{equation}
	By using Green's theorem we have
	\begin{equation*}
	\int_{\Omega}\left(\Delta(\sigma\partial_r u) u-\sigma\partial_r u\Delta u\right) 
	=\int_{\partial\Omega}\left(\partial_{\nu}(\sigma\partial_r u) u-\sigma\partial_r u\partial_{\nu} u\right) \, ,
	\end{equation*}
	from the boundary condition in \eqref{pb_principale_bis} and from \eqref{der_norm} we get
	\begin{equation*}
	\int_{\Omega}\left(\Delta(\sigma\partial_r u) u-\sigma\partial_r u\Delta u\right) 
	=-c^2\int_{\partial\Omega}\s\partial_{\nu}r  \, ,
	\end{equation*}
	from the divergence theorem we obtain
	\begin{equation}\label{10}
	\int_{\Omega}\left(\Delta(\sigma\partial_r u) u-\sigma\partial_r u\Delta u\right) 
	=-c^2n\int_{\Omega}\s'\, ,
	\end{equation}
	where we used the fact that $\Delta r=(n-1)\frac{\s'}{\s}$ (see \eqref{Laplace-Beltrami}).
	
	So from \eqref{10} and \eqref{America} we get
	\begin{equation}\label{America_bis}
	(2-n)\int_{\Omega}\sigma''\partial_r u u -2\int_{\Omega}\sigma' u - 2nk\int_{\Omega}\sigma' u^2+\int_{\Omega}\sigma\partial_r u=-c^2n\int_{\Omega}\s'\, ,
	\end{equation}
	Finally, we observe that from the divergence theorem and from the boundary condition in \eqref{pb_principale_bis} we have
	\begin{equation}\label{eq6}
	\int_{\Omega} \s \pr u = \int_{\Omega} g(\nabla u, \nabla ( \int_{0}^{r}\s(s)ds))
	= - \int_{\Omega} u \Delta (\int_{0}^{r}\s(s)ds)
	= -n \int_{\Omega} u \s'.
	\end{equation}
	Moreover 
	\begin{equation}\label{eq7}
	\int_{\Omega} \s'' \, u \pr u = \int_{\Omega}g(\nabla \s', \nabla (\frac{1}{2}u^{2}))
	= -  \frac{1}{2} \int_{\Omega} \Delta \sigma' \, u^{2}
	= - \frac{1}{2} \int_{\Omega} \frac{(\s'' \s^{n-1})'}{\s^{n-1}} u^{2}\, ,
	\end{equation}
	where we used the divergence theorem, the boundary condition in \eqref{pb_principale_bis} and the formula \eqref{Laplace-Beltrami}. By substituting \eqref{eq6} and \eqref{eq7} in \eqref{America_bis} we obtain \eqref{pre2}.
	
\end{proof}

\section{Proof of Theorems \ref{teo_princ}, \ref{teo_princ_bis} and \ref{teo_princ_ter}}\label{Principal_section}

We are now in position to prove the principal results of the paper.

\begin{proof}[Proof of Theorem \ref{teo_princ}]
%
From Lemma \ref{pre1} we have that $\Delta P(u)\geq 0$. Since $P(u)=c^2$ on $\partial\Omega$, then by the strong maximum principle either 
\begin{equation}\label{uguale}
P(u)\equiv c^2 \quad \text{in $\Omega$,} 
\end{equation}
or 
\begin{equation}\label{minore}
P(u)< c^2 \quad \text{in $\Omega$.} 
\end{equation}
By contradiction assume that condition \eqref{minore} is satisfied. 
Since \eqref{constants} is in force, we can multiply both members of \eqref{minore} by $\s'$ and integrating over $\Omega$ we have 
\begin{equation}\label{15}
\int_{\Omega}|\nabla u|^2\sigma'+\dfrac{2}{n}\int_{\Omega}u\sigma'+k\int_{\Omega}u^2\sigma'<c^2\int_{\Omega}\sigma'\, .
\end{equation}
From \eqref{pre2} we can substitute the second term, i.e.
\begin{equation}\label{16}
\dfrac{2}{n}\int_{\Omega}\sigma' u=-\int_{\Omega}\sigma' u+c^2\int_{\Omega}\sigma'+\dfrac{n-2}{2n}\int_{\Omega}\dfrac{(\sigma''\sigma^{n-1})'}{\sigma^{n-1}}u^2-2k\int_{\Omega}\sigma' u^2\, .
\end{equation}
Note that, from the divergence theorem and from the boundary condition in \eqref{pb_principale},
\begin{equation}\label{Queen}
\int_{\Omega}\s'\dive(u\nabla u)=-\int_{\Omega}\s''u\partial_r u\, ,
\end{equation}
and, since $\Delta u=-1-nku$ in $\Omega$, 
\begin{equation}\label{Queen_bis}
\int_{\Omega}\s'\dive(u\nabla u)=\int_{\Omega}\s'|\nabla u|^2-\int_{\Omega}\s' u - nk\int_{\Omega}\s' u^2\,.
\end{equation}
So from \eqref{Queen} and \eqref{Queen_bis} we get
\begin{equation}\label{17}
\int_{\Omega}\s'|\nabla u|^2=\int_{\Omega}\s' u + nk\int_{\Omega}\s'  u^2 -\int_{\Omega}\s''u\partial_r u\, ,
\end{equation}
Substituting \eqref{16} and \eqref{17} in  \eqref{15}, we obtain
\begin{equation*}
k(n-1)\int_{\Omega}\sigma' u^2-\int_{\Omega}\sigma'' u \partial_r u+\dfrac{n-2}{2n}\int_{\Omega}\dfrac{(\sigma''\sigma^{n-1})'}{\sigma^{n-1}}u^2<0\, .
\end{equation*}
Lastly, by using \eqref{eq7} we deduce
\begin{equation*}
k(n-1)\int_{\Omega}\sigma' u^2+\dfrac{n-1}{n}\int_{\Omega}\dfrac{(\sigma''\sigma^{n-1})'}{\sigma^{n-1}}u^2<0\, ,
\end{equation*}
i. e.
\begin{equation}\label{contrad}
\int_{\Omega}\left(k\sigma'+\dfrac{(\sigma''\sigma^{n-1})'}{n\sigma^{n-1}}\right)u^2<0\, .
\end{equation}

On the other hand we have that
\begin{equation*}
k\sigma'+\dfrac{(\sigma''\sigma^{n-1})'}{n\sigma^{n-1}}  = 0 \quad on \quad I,
\end{equation*}
since \eqref{warped_function} is in force. 
The latter clearly contradicts \eqref{contrad} and therefore \eqref{uguale} holds true and in particular $\Delta P(u)\equiv 0$ in $\Omega$. Hence from Lemma \ref{pre1} we get that \eqref{harm1} and \eqref{harm2} hold true. Condition \eqref{harm1} implies that $u$ is a solution to \eqref{Obata}, so from Lemma \ref{Obata_lemma} we conclude that $\Omega=B_g(p,\rho)$ and $u$ is a radial function with respect to $p$, where $p$ is the maximum of $u$ in $\Omega$. 


\end{proof}

\begin{proof}[Proof of Theorem \ref{teo_princ_bis}]
Argiung as before we obtain that
\begin{equation*}
\int_{\Omega}\left(k\sigma'+\dfrac{(\sigma''\sigma^{n-1})'}{n\sigma^{n-1}}\right)u^2<0\, ,
\end{equation*}
and this contradicts the compatibility assumption \eqref{comp}.

Therefore \eqref{uguale} holds true and in particular $\Delta P(u)\equiv 0$ in $\Omega$, hence from Lemma \ref{pre1} we get that \eqref{harm1} and \eqref{harm2} hold true. Condition \eqref{harm1} implies that $u$ is a solution to \eqref{Obata}, so from Lemma \ref{Obata_lemma} we conclude that $\Omega=B_g(p,\rho)$ and $u$ is a radial function with respect to $p$, where $p$ is the maximum of $u$ in $\Omega$. 

%
\end{proof}

\begin{proof}[Proof of Theorem \ref{teo_princ_ter}] We emphasize that the key ingredients in the proof of Theorem \ref{teo_princ_bis} are: the formula \eqref{eq2} and the Pohozaev-type inequality \eqref{pre2}. The analogue of this formulae holds also in the setting of model manifolds (i.e. warped product manifolds with a pole), indeed the analogue of formula \eqref{eq2} is shown in \cite[Lemma 6]{Ron} and the analogue of formula \eqref{pre2} can be obtained arguing as in \cite[Lemma 8]{Ron}. Therefore arguing as in  Theorem \ref{teo_princ_bis} we get that $\Omega$ is a metric ball, say $B_g(p,\rho)$ for some $\rho>0$ and $p\in M$ and $u$ is a radial function with respect to $p$. Moreover, we know that $u$ satisfies the following equation
	\begin{equation}\label{Hessian_model}
	\nabla^2 u=- \left(\dfrac{1}{n}+ku\right)g \quad \text{in $\Omega$}\, .
	\end{equation}
	To prove the Theorem we argue as in the proof of Theorem 3 in \cite{Ron}. Let $\hat\rho:=d_g(o,\partial B_g(p,\rho))$, where $d_g$ is the intrinsic distance in the model manifold. Take $B_g(o,\hat{\rho})\subset B_g(p,\rho)$. To conclude the proof it is enough to show that 
	\begin{equation}\label{=_balls}
	B_g(o,\hat{\rho})=B_g(p,\rho)\, .
	\end{equation}
	To prove \eqref{=_balls} we consider $u|_{B_g(o,\hat\rho)}:B_g(o,\hat\rho)\rightarrow\mathbb{R}$ which solves \eqref{Hessian_model} in $B_g(o,\hat\rho)$. Take $x\in B_g(o,\hat\rho)$ then, by the results in Appendix \ref{Appendix} , there exists a minimizing and unit speed geodesic $\gamma\subset B_g(o,\hat\rho)$ from $o$ to $x$. Let $y(t):=u\circ \gamma(t)$ and we note that, along $\gamma$, equation \eqref{Hessian_model} implies: 
	\begin{align*}
	y''(t)&=\dfrac{d^2}{dt^2}(u\circ\gamma)(t) \\
	&=\dfrac{d}{dt}g(\nabla u(\gamma(t)),\dot{\gamma}(t)) \\
	&=g(D_{\dot{\gamma}}\nabla u(\gamma(t)),\dot{\gamma}(t))+g(\nabla u(\gamma(t)),D_{\dot{\gamma}}\dot{\gamma}(t))\\
	&=g((D_{\dot{\gamma}(t)}\nabla u)(\gamma(t)),\dot{\gamma}(t))\\
	&=\nabla^2(u)\mid_{\gamma(t)}(\dot{\gamma}(t),\dot{\gamma}(t))\\
	&=-\frac{1}{n}-ky(t)\, .
	\end{align*}
	The solutions to 
	$$
	y''(t)=-\dfrac{1}{n}-ky(t)
	$$
	are given by 
	\begin{equation}
	y(t)=\begin{cases}
	\alpha\sin(\sqrt{k}t)+\beta\cos(\sqrt{k}t)-\frac{1}{kn} &\mbox{if } k>0 , \\ -\frac{t^2}{2n}+\alpha t+\beta  &\mbox{if } k=0 ,\\
	\alpha e^{\sqrt{-k}t}+\beta e^{-\sqrt{-k}t}-\frac{1}{kn} &\mbox{if } k<0
	\, .
	\end{cases}
	\end{equation}
	where $\alpha$, $\beta\in\mathbb{R}$. Now taking $t=r(x)=d_g(x,o)$, the intrinsic distance from the pole $o$, we get 
	\begin{equation}\label{sol_model}
	u(x)=u\circ \gamma(r(x))=y(r(x))=\begin{cases}
	\alpha\sin(\sqrt{k}r(x))+\beta\cos(\sqrt{k}r(x))-\frac{1}{kn} &\mbox{if } k>0 , \\ -\frac{r(x)^2}{2n}+\alpha r(x)+\beta  &\mbox{if } k=0 , \\
	\alpha e^{\sqrt{-k}r(x)}+\beta e^{-\sqrt{-k}r(x)}-\frac{1}{kn} &\mbox{if } k<0
	\, ,
	\end{cases}
	\end{equation}
	which is a radial function with respect to $o$, hence $p=o$ and from the fact that $u>0$ in $B_g(p,\rho)$ and from the boundary condition $u=0$ on $\partial B_g(p,\rho)$ we get that $\hat\rho=\rho$. Hence we have shown \eqref{=_balls}.
	
	We observe that, in a system of normal coordinates in the pole $o$ of the model we have that $r(x)=|x|$, hence 
	$$
	\nabla u(x)=\begin{cases}
	\alpha\sqrt{k}\cos(\sqrt{k}|x|)\frac{x}{|x|}-\beta\sqrt{k}\sin(\sqrt{k}|x|)\frac{x}{|x|} &\mbox{if } k>0 , \\ -\frac{x}{n}+\alpha \frac{x}{|x|}  &\mbox{if } k=0 , \\
	\alpha \sqrt{-k}e^{\sqrt{-k}|x|}\frac{x}{|x|}-\beta\sqrt{-k} e^{-\sqrt{-k}|x|}\frac{x}{|x|} &\mbox{if } k<0
	\, ,
	\end{cases}
	$$
	which is not a $C^1$ function in the pole (i.e. $|x|=0$) unless 
	\begin{equation}\label{value_alpha}
	\alpha=0 \quad \text{ if $k\geq 0$} \quad \text{and} \quad \alpha=\beta \quad \text{ if $k< 0$}\, .
	\end{equation}
	Moreover, from the condition $u(\rho)=0$ we get 
	\begin{equation}\label{value_beta}
	\beta=\begin{cases}
	\frac{1}{kn\cos(\sqrt{k}\rho)} &\mbox{if } k>0 , \\ \frac{\rho^2}{2n}  &\mbox{if } k=0 ,\\
	\frac{1}{2kn\cosh(\sqrt{-k}\rho)} &\mbox{if } k<0
	\, .
	\end{cases}
	\end{equation}
	Summing up, putting \eqref{value_alpha} and \eqref{value_beta} in \eqref{sol_model}, we get 
	\begin{equation}\label{sol_radial}
	u(r)=\begin{cases}
	\frac{\cos(\sqrt{k}r)}{kn\cos(\sqrt{k}\rho)}-\frac{1}{kn} &\mbox{if } k>0 , \\ \frac{\rho^2}{2n}-\frac{r^2}{2n} &\mbox{if } k=0 , \\
	\frac{\cosh(\sqrt{-k}r)}{kn\cosh(\sqrt{-k}\rho)}-\frac{1}{kn} &\mbox{if } k<0\, .
	\end{cases}
	\end{equation}
	Now, we recall that if $f:(M,g)\rightarrow\mathbb{R}$ is a smooth radial function on a model manifold, then its Hessian takes the following expression
	$$
	\nabla^2 f=f''dr\otimes dr+f'\s\s'g_{\mathbb{S}^{n-1}}.
	$$
	Using this expression with the function $u$ given by \eqref{sol_radial} we get
	$$
	\nabla^2 u(r)=\begin{cases}
	-\frac{\cos(\sqrt{k}r)}{n\cos(\sqrt{k}\rho)}dr\otimes dr-\frac{\sqrt{k}\sin(\sqrt{k}r)}{kn\cos(\sqrt{k}\rho)}\sigma(r)\sigma'(r)g_{\mathbb{S}^{n-1}} &\mbox{if } k>0 , \\ -\frac{1}{n}dr\otimes dr - \frac{r}{n}\sigma(r)\sigma'(r)g_{\mathbb{S}^{n-1}} &\mbox{if } k=0 , \\
	-\frac{\cosh(\sqrt{-k}r)}{n\cosh(\sqrt{-k}\rho)}dr\otimes dr+\frac{\sqrt{-k}\sinh(\sqrt{-k}r)}{kn\cosh(\sqrt{-k}\rho)}\sigma(r)\sigma'(r)g_{\mathbb{S}^{n-1}} &\mbox{if } k<0\, ;
	\end{cases}
	$$
	and using \eqref{Hessian_model} we can prove \eqref{metric}. Indeed, for $r\in(0,\rho)$,
	
	\medskip 
	
	\noindent	$\bullet$ 
	if $k>0$ 
	\begin{equation*}
	-\dfrac{\sqrt{k}\sin(\sqrt{k}r)}{kn\cos(\sqrt{k}\rho)}\sigma(r)\sigma'(r)g_{\mathbb{S}^{n-1}}= -\dfrac{\cos(\sqrt{k}r)}{n\cos(\sqrt{k}\rho)}\sigma^2(r)g_{\mathbb{S}^{n-1}}\, ,
	\end{equation*}
	which implies
	$$
	\frac{\sin(\sqrt{k}r)}{\sqrt{k}}\sigma(r)\sigma'(r)=\cos(\sqrt{k}r)\sigma^2(r)\, ,
	$$
	i.e. 
	\begin{equation}\label{sin}
	\sigma(r)=\dfrac{\sin(\sqrt{k}r)}{\sqrt{k}}\, .
	\end{equation}

	\smallskip
	
	\noindent	$\bullet$ 
	If $k=0$ 
	$$
	-\dfrac{1}{n}r\s(r)\s'(r)g_{\mathbb{S}^{n-1}}=-\dfrac{1}{n}\s^2(r)g_{\mathbb{S}^{n-1}}\, ,
	$$
	which implies that 
	$$
	r\s(r)\s'(r)=\s^2(r)\, ,
	$$
	i.e. 
	\begin{equation}\label{r}
	\sigma(r)=r\, .
	\end{equation}
	
	\smallskip
	
	\noindent	$\bullet$
	 If $k<0$
		\begin{equation*}
		\dfrac{\sqrt{-k}\sinh(\sqrt{-k}r)}{kn\cosh(\sqrt{-k}\rho)}\sigma(r)\sigma'(r)g_{\mathbb{S}^{n-1}}=-\dfrac{\cosh(\sqrt{-k}r)}{n\cos(\sqrt{-k}\rho)}\sigma^2(r)g_{\mathbb{S}^{n-1}}\, ,
		\end{equation*}
		which implies that 
		$$
		\dfrac{\sinh(\sqrt{-k}r)}{\sqrt{-k}}\sigma(r)\sigma'(r)=\cosh(\sqrt{-k}r)\sigma^2(r)\, ,
		$$
		i.e. 
		\begin{equation}\label{sinh}
		\sigma(r)=\dfrac{\sinh(\sqrt{-k}r)}{\sqrt{-k}}\, ;
		\end{equation}
		where the constants in \eqref{sin}, \eqref{r} and in \eqref{sinh} are chosen in such a way that $\sigma'(0)=1$ (see \cite[Definition 2]{Ron}). So \eqref{metric} follows.
\end{proof}

\begin{remark}[An alternative proof]
	{\rm
One can prove Theorem \ref{teo_princ_ter} arguing in the following way: from Lemma \ref{pre1} we know that 
\begin{equation}\label{1}
\nabla^2 u=-\left(\dfrac{1}{n}+ku\right)g  \quad \text{in $\Omega$,}
\end{equation}
and 
\begin{equation}\label{2}
\ric_M(\nabla u,\nabla u)=(n-1)k|\nabla u|^2  \quad \text{in $\Omega$,}
\end{equation}
hold true. As in the proof of the Theorem one can see that \eqref{1} implies that $u$ is given by \eqref{sol_radial} and that $\Omega=B_{g}(o,\rho)$. Now, being $(M,g)$ a model manifold and since $u$  is a radial function we get that
$$
\ric_M(\nabla u,\nabla u)=-(n-1)\dfrac{\sigma''(r)}{\sigma(r)}|\nabla u|^2\, .
$$
So from \eqref{2} we conclude that 
$$
-\dfrac{\sigma''(r)}{\sigma(r)}=k \quad \text{for $r\in(0,\rho)$}\, ,
$$
and this implies that $\sigma$ is given by \eqref{warped_function}. Now the hypothesis on $\sigma$ that makes the metric $g$ smooth in the pole are 
$$
\sigma^{(2k)}(0)=0 \quad \text{ for all $k=0,1,2,\dots$} \qquad \text{and} \qquad \sigma'(0)=1\, ,
$$
and hence \eqref{metric} follows.}
\end{remark}


\bigskip

\appendix

\section{About the star-shapedness property of balls}\label{Appendix}

Let $(M,g)$ be a Riemannian manifold without boundary (not necessarily complete). We will use the following notations: $\ell_g$ is the lenght of a $C^1$-piecewise curve in $(M,g)$, $d_g$ is the intrinsic distance in $(M,g)$ and we will denote by $B_g(p,R)$ the open ball centered at $p$ and of radius $R$. We recall that 
$$
\overline{B}_g(p,R)=\overline{B_g(p,R)} \quad \text{ for all $p\in M$ and $R>0$,}
$$
i.e., the closed ball centered at $p$ and of radius $R$ coincides with the closure of the open ball centered at $p$ and of radius $R$. 
The main results of this Appendix 
deal with the star-shapedness property of balls. The first result is 

\begin{lemma}\label{lemmaGeod1}
Let $\Omega \subset M$ be a domain such that $ \overline{\Omega}$ is compact and consider $B_g(p,R)\subseteq \Omega$. Then for all $q\in B_g(p,R)$ there exists a minimizing geodesic $\gamma:[0,1]\rightarrow B_g(p,R)$ such that $\gamma(0)=p$ and $\gamma(1)=q$.
\end{lemma}

\begin{proof}
Let $q\in B_g(p,R)$ and consider a minimizing sequence $\{\gamma_k\}$ of constant speed piecewise regular curves $\gamma_k:[0,1]\rightarrow\overline{B}_g(p,R)$, such that $\gamma_k(0)=p$, $\gamma_k(0)=q$ and 
$$
\ell_g(\gamma_k)\rightarrow d_g(p,q)<R\, .
$$
Since $d_g(p,q)<R$ we may and do suppose that $\gamma_k([0,1])\subset \overline{B}_g(p,R)$, for every $k\geq 1$. Since $\{\ell_g(\gamma_k)\}$ is uniformly bounded we have that the constant speeds of $\gamma_k$ are bounded by a fixed constant $C>0$ and so
$$
d_g(\gamma_k(s),\gamma_k(t))\leq \ell_g(\gamma_k|_{[s,t]})\leq C(t-s)\, , \quad \text{ for every $k\geq 1$ and $s,t\in[0,1]$ with $s<t$.}
$$
The latter implies the the sequence $\{\gamma_k\}$ is a sequence of equilipschitz maps from $[0,1]$ to the compact metric space $(\overline{B}_g(p,R),d_g)$ and therefore, by Ascoli-Arzel\`a Theorem, we can find a subsequence which converges uniformly to a curve $\gamma:[0,1]\rightarrow \overline{B}_g(p,R)$ with $\gamma(0)=p$, $\gamma(1)=q$ and $\gamma$ is Lipschitz and minimizing. Summing up, we prove that $\gamma:[0,1]\rightarrow \overline{B}_g(p,R)$ is a Lipschitz curve which minimize the so-called induced length:
$$
\ell_d(\gamma):=\sup\left\{ \sum_{i=1}^n d_g(\gamma(t_{i-1}),\gamma(t_i)) \, : \, n\in\mathbb{N}\, , 0=t_0<t_1<\dots<t_n=1\right\}\, .
$$
Now, since $\gamma$ is absolutely continuous, from \cite[Proposition 3.7]{Burtscher} (see also \cite[Section 2.7]{BBY}) we know that 
$$
\ell_g(\gamma)=\int_{0}^1|\gamma'(t)|_g\, dt
$$
is well-defined\footnote{We recall that given $M$ a connected smooth manifold with a smooth Riemannian metric $g$ and given $\gamma:[0,1]\rightarrow M$ a piecewise smooth curve, then its length is defined as 
$$
\ell_g(\gamma):=\int_0^1|\gamma'(t)|_g\, dt \, .
$$
If $\gamma:[0,1]\rightarrow M$ is just a continuous path in $M$, then its length is defined as
$$
\ell_d(\gamma):=\sup\left\{ \sum_{i=1}^n d_g(\gamma(t_{i-1}),\gamma(t_i)) \, : \, n\in\mathbb{N}\, , 0=t_0<t_1<\dots<t_n=1\right\}\, ,
$$
where $d_g$ is the intrinsic distance in $M$. A well-known Theorem (see e.g. \cite[Theorem 2.2]{Burtscher}) asserts that: if $M$ is a connected manifolds with a Riemannian metric $g$. Then, for all piecewise smooth curve $\gamma$,
$$
\ell_g(\gamma)=\ell_d(\gamma)\, .
$$
The point is that $\ell_d$ is defined for a larger class of curves. One can show (see e.g. \cite[Proposition 3.7]{Burtscher}) that if $M$ is a connected manifold with a continuous Riemannian metric $g$, then for any absolutely continuous path $\gamma:[0,1]\rightarrow M$ the derivative $\gamma'$ exists a.e. and $|\gamma'|_g\in L^1$. In particular 
$$
\ell_d(\gamma)=\int_0^1|\gamma'(t)|_g\, dt\, ,
$$ 
is a well-defined length for any absolutely continuous path $\gamma$.}. 
Hence $\ell_d(\gamma)=\ell_g(\gamma)$ and $\gamma$ is a minimizer also for the functional $\ell_g$. Now the thesis follows from the following general fact: if the Riemannian metric $g$ is $C^1$, then any locally Lipschitz and minimizing curve is a geodesic (see e. g. \cite[Section 2]{SS}).

\end{proof}

Thanks to Lemma \ref{lemmaGeod1} we can prove the next Proposition, in which we show the star-shapedness property of balls.

\begin{proposition}\label{LemmaGeod2}
	Let $\Omega \subset M$ be a domain such that $ \overline{\Omega}$ is compact
	and consider $B_g(p,R)\subseteq \Omega$. Then for all $q\in \partial B_g(p,R)$ there exists a minimizing geodesic $\gamma:[0,1]\rightarrow \overline{B_g(p,R)}$ such that $\gamma(0)=p$, $\gamma(1)=q$ and $ \gamma(t) \in B_g(p,R)$ for any $t \in [0,1)$.
\end{proposition}
\begin{proof}
 Since $q\in \partial B_g(p,R)$  we can find a sequence $\{q_k\}$ of points of $B_g(p,R)$ such that $ q_k \to q$ in  $(\overline{B}_g(p,R),d_g)$.  By Lemma \ref{lemmaGeod1} there are minimizing geodesics $\gamma_n :[0,1]\rightarrow {B_g(p,R)}$ such that $\gamma(0)=p$, $\gamma_k(1)=q_k$, for any $ k \geq 1$. 
The same argument used in the second proof of Lemma \ref{lemmaGeod1} implies that the sequence $\{\gamma_k\}$ admits a subsequence which converges uniformly to a Lipschitz curve $\gamma:[0,1]\rightarrow \overline{B}_g(p,R)$ with $\gamma(0)=p$ and $\gamma(1)=q$ (since $ q_k \to q$ in $\overline{B_g(p,R)}$ implies $\gamma_k(1) =q_k  \to q $ while the convergence of $\{\gamma_k \} $ implies $ \gamma_k(1) \to \gamma(1)$.) 
Therefore $\gamma$ is a Lipschitz minimizing curve and so it is a geodesic by \cite[Section 2]{SS}. 
\end{proof}

\end{document}